\documentclass[11pt]{article}
\usepackage{makeidx}
\usepackage[centertags]{amsmath}
\usepackage{amsthm}
\usepackage{newlfont}
\usepackage{amsmath}
\usepackage{amsfonts,amsmath,latexsym}
\usepackage{mathrsfs}
\usepackage{amssymb}
\usepackage{amsbsy}
\usepackage{indentfirst}
\usepackage{mathrsfs}
\usepackage{graphicx}
\usepackage{epstopdf}
\usepackage{subfigure}
\usepackage{caption}
\usepackage{grffile}
\usepackage{float}
\usepackage{amssymb, amsmath}
\usepackage{xcolor}
\usepackage{enumitem}
\usepackage{CJK, CJKnumb, CJKulem, verbatim}

\titlepage
\newlength{\defbaselineskip}
\newcommand{\be}{\begin{eqnarray}}
\newcommand{\ee}{\end{eqnarray}}
\newcommand{\bestar}{\begin{eqnarray*}}
\newcommand{\eestar}{\end{eqnarray*}}

\newcommand{\ignore}[1]{}
{} \theoremstyle{plain}
\newtheorem{thm}{Theorem}[section]
\newtheorem{cor}{Corollary}[section]
\newtheorem{lemma}{Lemma}[section]

\newtheorem{exam}{Example}[section]
\theoremstyle{definition}

\newtheorem{rem}{Remark}[section]
\numberwithin{equation}{section}

\def\>{\geq}
\def\<{\leq}

\allowdisplaybreaks
\graphicspath{{./graphics/}}

\begin{document}

\title{ Strong limit theorems for step-reinforced random walks}

\author{
Zhishui Hu \thanks{Email: huzs@ustc.edu.cn} ~~~~and~~~~ Yiting Zhang \thanks{Email: zyt1999@mail.ustc.edu.cn} \\
{\small Department of Statistics and Finance, School of Management} \\
{\small University of Science and Technology of China}\\
{\small Hefei 230026, China}}

\date{}

\maketitle

\begin{abstract}

A step-reinforced random walk is a discrete-time non-Markovian process with long range memory.
At each step,   with a fixed probability $p$, the positively step-reinforced random walk  repeats one of its preceding steps chosen uniformly at random,
and  with  complementary probability $1-p$, it has an independent increment.
The negatively step-reinforced random walk follows the same reinforcement algorithm but when a step is repeated its sign is also changed.
Strong laws of large numbers and strong invariance principles  are established for positively and negatively step-reinforced random walks in this work.
Our approach relies on  two general theorems on invariance
principle for martingale difference sequences and a truncation argument. As by-products of our main results,  the law of iterated logarithm
and the functional central limit theorem are also obtained for step-reinforced random walks.

 \vskip 0.2cm

\end{abstract}

\section{Introduction}

Let $\{X_n, n\ge 1\}$ be a sequence  of i.i.d. random variables and
let $\{\varepsilon_n, n \ge 2\}$ be a sequence of i.i.d. Bernoulli variables with parameter $p\in [0,1]$.
For each $n\ge 2$, let  $U_n$ be a random variable uniformly distributed on $ \{ 1,2, \cdots ,n-1 \} $.   We always assume that all the above random variables $X_1, X_2, \varepsilon_2, U_2, X_3, \varepsilon_3, U_3,\cdots$ are independent of each other.
 Set first $\hat{X}_1=X_1$ and then define recursively for $n\ge 2$,
\be
\label{defhatX}
\hat{X}_n :=
\left\{
\begin{aligned}
	& X_n, & \text{if } \varepsilon_n = 0, \\
	& \hat{X}_{U_n}, & \text{if } \varepsilon_n =1.
\end{aligned}
\right.
\ee
The sequence of the partial sums
$$ \hat{S}(n) := \hat{X}_1 + \cdots + \hat{X}_n , \qquad n \ge 1 $$
is called a positively step-reinforced random walk or noise reinforced random walk.
In the above construction, at each time $n\ge 2$, with probability $p$, a positively step-reinforced random walk chooses one of the past steps uniformly at random
and repeats it, and otherwise (i.e., with probability $1-p$) it has an independent increment.
The parameter $p$ is called the reinforcement parameter.

The algorithm (\ref{defhatX}) was introduced by Simon \cite{S55} to explain the appearance of
power laws  in a wide range of empirical data.
Simon's algorithm is closely related to preferential attachment dynamics in complex networks.
 For more details, we refer to \cite{AB2002} and references therein.

When $X_1$ has the Rademacher law,  i.e., $\mathbb{P}(X_1=1)=\mathbb{P}(X_1=-1)=1/2$, ~$\hat{S}$
is a so-called elephant random walk. The elephant random walk was introduced by Sch\"{u}tz and Trimper \cite{SCH2004}. It
is a discrete-time nearest-neighbor random walk on integers $\mathbb{Z}$, which has a memory about the whole
past. It can be depicted as follows:
 Fix some $q\in [0,1]$ which is called the memory parameter. An elephant makes a first step in $\{-1,+1\}$. At each time $n\ge 2$, the elephant chooses
one of the past steps uniformly at random, then either repeats it with probability $q$, or makes
a step in the opposite direction with probability $1-q$.
The asymptotic behaviour of the elephant random walk has been extensively studied; see e.g. \cite{BB2016,BE2017,Col2017b,Col2017a,KT2019,KU2016}.
When $X_1$  has a symmetric stable distribution, $\hat{S}$ is a so-called shark random swim which has been studied in \cite{BU2018}.
When the distribution of $X_1$ is arbitrary, $\hat{S}$ has been studied in \cite{B2020B,B2020C,B2021A,B2021B}.

Throughout this paper, we write
$m_i=\mbox{E}(X_1^i)$ for $i=1,2$  and $\sigma^2=\mbox{Var}(X_1)=m_2-m_1^2.$
 Supposing that $\mathbb{E}(X_1^2)<\infty$, Bertoin \cite{B2021A} proved that the asymptotic behavior of the  positively step-reinforced random walk exhibits a phase transition similar to those for the elephant random walk and the shark random swim:
   in the super-diffusive regime ($p\in (1/2,1)$),~  $n^{-p}(\hat{S}_n-m_1n)$
converges in $L^2(\mathbb{P})$ to some non-degenerate random variable $L$;  in the diffusive regime ($p\in (0,1/2)$),
we have
\be
\Big\{\frac{\hat{S}([nt])-m_1nt}{\sigma\sqrt{n}},~ t\ge 0\Big\}\Longrightarrow  \big \{\hat{B}(t), ~t\ge 0\big\},  \label{wc1}
\ee
where $\Longrightarrow$ denotes weak convergence with respect to the Skorohod topology (see e.g. \cite{Bi1968}),   and $\hat{B}=\{\hat{B}(t), ~t\ge 0\}$ is a noise reinforced Brownian motion
with parameter $p$, which is a centered Gaussian process with covariance function
\bestar
\mathbb{E}(\hat{B}(s)\hat{B}(t))=\frac{t^ps^{1-p}}{1-2p}, ~~0\le s\le t.
\eestar
Let $\{W(t), t\ge 0\}$ be a standard Brownian motion, then
 $\hat{B}$ has
the same law as
\bestar
\Big\{\frac{t^{p}}{\sqrt{1-2p}} W( t^{1-2p}),~ t\ge 0\Big\}.
\eestar
In the critical regime  ($p=1/2$), Bertengui and Rosales-Ortiz \cite{BR2022} obtained that
\be
\Big\{\frac{\hat{S}([n^t])-m_1n^t}{\sigma \sqrt{n^t\log n}},~t\ge 0\Big\}  \Longrightarrow  \{W(t),~t\ge 0\}.   \label{wc2}
\ee

Recall that
 when $X_1$ has the Rademacher law,  $\hat{S}$ is an
elephant random walk with memory parameter $q=(p+1)/2\in [1/2,1]$ (see \cite{KU2016}).
To get  the corresponding generalisation of the
elephant random walk with memory parameter in the remaining range,
Bertoin \cite{B2020A} introduced the following counterbalancing algorithm. Set $\check{X}_1=X_1$ and define recursively for $n\ge 2$,
\be
\label{defcheckX}
\check{X}_n :=
\left\{
\begin{aligned}
	& X_n, & \text{if } \varepsilon_n = 0, \\
	& -\check{X}_{U_n}, & \text{if } \varepsilon_n = 1.
\end{aligned}
\right.
\ee
The process
$$ \check{S}(n) := \check{X}_1 + \cdots + \check{X}_n , \qquad n \ge 1 $$
is called a negatively step-reinforced random walk or counterbalanced random walk.
 When $X_1$ has the Rademacher law, $(\check{S}_n)_{n\ge 1}$ is an
elephant random walk with memory parameter $q=(1-p)/2\in [0, 1/2]$.
Bertoin \cite{B2020A} established the central limit theorem for $\check{S}_n$.
Bertengui and Rosales-Ortiz \cite{BR2022} obtained the functional central limit theorem:
for $p\in (0,1)$, if $m_1=0$ and $\mathbb{E}(X_1^2)<\infty$, then
\be
\Big\{\frac{\check{S}([nt])}{\sigma\sqrt{n}},~t\ge 0\Big\} \Longrightarrow  \{\check{B}(t), ~t\ge 0\},
  \label{wc3}
\ee
where $\check{B}=\{\check{B}(t), ~t\ge 0\}$ is  a centered Gaussian process with covariance function
\bestar
\mathbb{E}(\check{B}(s)\check{B}(t))=\frac{s^{1+p}t^{-p}}{2p+1}, ~~0\le s\le t.
\eestar
It is clear that
 $\check{B}$ has
the same law as
\bestar
\Big\{\frac{t^{-p}}{\sqrt{2p+1}} W( t^{2p+1}),~ t\ge 0\Big\}.
\eestar
Furthermore,
Bertengui and Rosales-Ortiz \cite{BR2022} obtained the joint functional central limit theorems for $(S_n, \hat{S}_n,\check{S}_n)$.

In the literature,
there are very few results on strong limit theorems for $\hat{S}_n$ and  $\check{S}_n$.
In this direction, Bertengui and Rosales-Ortiz \cite{BR2022} established the strong laws of large numbers  under the assumption of finite second moment. More
precisely, they showed that if $\mathbb{E}(X_1^2)<\infty$, then $n^{-1}\hat{S}_n\rightarrow m_1~a.s.$ and $n^{-1}\check{S}_n\rightarrow (1-p)m_1/(1+p)~a.s.$
In fact, the hypothesis   $\mathbb{E}(X_1^2)<\infty$ can be
weakened.
In this paper, we begin by providing the strong laws of large numbers  under the minimal moment condition $\mathbb{E}(|X_1|)<\infty$.

Note that for $p=0$, both $\hat{S}$ and $\check{S}$ are just the classical random walk with i.i.d. steps,
and for $p=1$, we have $\hat{S}_n=nX_1$ for all $n\ge 1$ and the distribution of $\check{S}$ can be expressed explicitly in terms of Eulerian numbers (see e.g. \cite{B2020A}).
Our attention is focussed on the case  $p\in (0,1)$.

\begin{thm}
	\label{LLN}
	If $\mathbb{E}|X_1|<\infty$, then for any $ p \in (0,1) $, we have
	\bestar	
	\frac{\hat{S}(n)}{n} \rightarrow m_1~~a.s. ~~~~\mbox{and}~~~~\frac{\check{S}(n)}{n} \rightarrow \frac{1-p}{1+p}m_1~~a.s.
	\eestar
\end{thm}

The main aim of this paper is to establish the strong invariance principles for  $\hat{S}$ and $\check{S}$ when $\mathbb{E}(X_1^2)<\infty$. At the same time, we also get
the uniform weak invariance principles.

\begin{thm} \label{th4.1} Suppose that $\mathbb{E}(X_1^2)<\infty$. If $p\in (0, 1/2)$, then there exists  a richer probability space with a standard Brownian motion $\{W(t), t\ge 0\}$ such that
\be
 \frac{1}{\sqrt{n^{1-2p}\log\log n}}\Big|\frac{\sqrt{1-2p}}{\sigma}\frac{\hat{S}(n)-m_1n}{n^{p}}-W(n^{1-2p})\Big|{\rightarrow} 0~~a.s.,  \label{th4.1eq1}
\ee
and
\be
 \max_{0\le t\le 1}\Big|\frac{\sqrt{1-2p}}{\sigma}\frac{\hat{S}([nt])-m_1[nt]}{\sqrt{n}}-\frac{t^{p} W( (nt)^{1-2p})}{\sqrt{n^{1-2p}}}\Big|\stackrel{P}{\rightarrow} 0.
 \label{th4.1eq2}
\ee
If $p=1/2$, then there exists  a richer probability space with a standard Brownian motion $\{W(t), t\ge 0\}$ such that
\be
\frac{1}{\sqrt{\log n\log\log \log n}} \Big|\frac{\hat{S}(n)-m_1n}{\sigma \sqrt{n}}-{W}(\log n)\Big|\rightarrow 0~~a.s.,
\label{th4.1eq7}
\ee
and
\be
\frac{1}{\sqrt{\log n}} \max_{1/n \le t\le 1}\Big|\frac{\hat{S}([nt])-m_1[nt]}{\sigma \sqrt{[nt]}}-{W}(\log (nt))\Big| \stackrel{P}{\rightarrow} 0.
\label{th4.1eq8}
\ee
\end{thm}

\begin{thm} \label{th4.2} Suppose that $\mathbb{E}(X_1^2)<\infty$ and $p\in (0,1)$. Write
\be
\check{\mu}=\frac{(1-p)m_1}{1+p},~~\check{\sigma}^2=m_2-\Big(\frac{(1-p)m_1}{1+p}\Big)^2. \label{notation1}
\ee
 Then there exists  a richer probability space with a standard Brownian motion $\{W(t), t\ge 0\}$ such that
\bestar
 \frac{1}{\sqrt{n^{2p+1}\log\log n}}\Big|\frac{\sqrt{2p+1}}{\check{\sigma}}\frac{\check{S}(n)-\check{\mu}n}{n^{-p}}-W(n^{2p+1})\Big|{\rightarrow} 0~~a.s.,
\eestar
and
\be
 \max_{0\le t\le 1}\Big|\frac{\sqrt{2p+1}}{\check{\sigma}}\frac{\check{S}([nt])-\check{\mu}[nt]}{\sqrt{n}}-\frac{t^{-p}W((nt)^{2p+1})}{\sqrt{n^{2p+1}}}\Big|\stackrel{P}{\rightarrow} 0.  \label{unif1}
\ee
\end{thm}

\begin{rem} Suppose that $\mathbb{E}(X_1^2)<\infty$ and $p\in (0,1)$.
For any $N\in \mathbb{N}$, a similar argument as in the proof of Theorem  \ref{th4.2} gives that (\ref{unif1}) still holds  by replacing $\max\limits_{0\le t\le 1}$ by $\max\limits_{0\le t\le N}$.
From this it follows immediately that  for any $N\in \mathbb{N}$,
\bestar
\Big\{\frac{\check{S}([nt])-\check{\mu}[nt]}{\check{\sigma}\sqrt{n}},~0\le t\le N\Big\} \Longrightarrow  \Big\{\frac{t^{-p}W(t^{2p+1})}{\sqrt{2p+1}}, ~ 0\le t\le N\Big\}.
\eestar
Hence (see \cite{Bi1968})
\be
\Big\{\frac{\check{S}([nt])-\check{\mu}[nt]}{\check{\sigma}\sqrt{n}},~t\ge 0\Big\} \Longrightarrow  \Big\{\frac{t^{-p}W(t^{2p+1})}{\sqrt{2p+1}}, ~  t\ge 0\Big\}.  \label{add2}
\ee
This implies (\ref{wc3}) when $m_1=0$.  Similarly, we can obtain the functional central limit theorems (\ref{wc1}) and (\ref{wc2}) from Theorem \ref{th4.1}.
Moreover, it follows from (\ref{th4.1eq8}) that for $p=1/2$,
\be
\frac{1}{\sqrt{\log n}} \max_{0 \le t\le 1}\Big|\frac{\hat{S}([nt])-m_1[nt]}{\sigma \sqrt{n}}-\sqrt{t}{W}(\log n)\Big| \stackrel{P}{\rightarrow} 0.  \label{add}
\ee
The proof of (\ref{add})  will be given in Section \ref{sect4}.
Using similar argument as in the proof of (\ref{add2}) gives that for $p=1/2$, we also have
\bestar
\Big\{\frac{\hat{S}([nt])-m_1[nt]}{\sigma \sqrt{n\log n}},~t\ge 0\Big\}  \Longrightarrow  \{\sqrt{t} W(1),~t\ge 0\}.
\eestar
\end{rem}

By applying Theorems \ref{th4.1} and \ref{th4.2} and  the law of the iterated logarithm for $\{W(t), t \ge 0\}$,  we can also get the laws of the iterated logarithm for $\hat{S}_n$ and  $\check{S}_n$, which are
stated as follows without proof.

\begin{cor}  Suppose that $\mathbb{E}(X_1^2)<\infty$. Define $\check{\mu}$ and $\check{\sigma}^2$ as in (\ref{notation1}).
\begin{enumerate}

\item For $p\in (0, 1/2)$, we have
\bestar
\limsup_{n\rightarrow\infty}\frac{\hat{S}(n)-m_1n}{\sqrt{2n\log\log n}}=\frac{\sigma}{\sqrt{1-2p}}~~~~a.s.
\eestar
\item  For $p=1/2$, we have
\bestar
\limsup_{n\rightarrow\infty}\frac{\hat{S}(n)-m_1n}{\sqrt{2n\log n \log \log\log n}}=\sigma~~~~a.s.
\eestar
\item For $p\in (0,1)$, we have
\bestar
\limsup_{n\rightarrow\infty}\frac{\check{S}(n)-\check{\mu}n}{\sqrt{2n\log\log n}}=\frac{\check{\sigma}}{\sqrt{2p+1}}~~~~a.s.
\eestar
\end{enumerate}
\end{cor}

We will prove our main results  by combining the martingale method and a truncation argument.
The martingale method has been successfully applied to proving limit theorems for the elephant random walk; see e.g. \cite{BE2017,Col2017a,Fan2022}. 
However, in the case of the elephant random walk, some properties of the martingale rely on the boundedness of $X_1$. In order to get the corresponding limit theorems for $\hat{S}_n$ and $\check{S}_n$ with general steps, it is necessary to adopt the truncation technique.  A special truncation argument was used in \cite{B2021A, BR2022} 
to obtain the functional central limit theorems for $\hat{S}_n$ and $\check{S}_n$ (see Section 4.3 in \cite{B2021A} and Section 4 in \cite{BR2022}). But this  is not applicable to strong limit theorems. In this work, we find a different truncation argument.

The rest of this paper is organized as follows. Section \ref{sect2} presents two general theorems on invariance principle for  martingale difference sequences.
The proofs of the results in Section \ref{sect2} are offered in Section \ref{sect3}. Section \ref{sect4}  is devoted to the proofs of Theorems \ref{LLN}$-$\ref{th4.2}.

\section{Invariance principles for  martingale difference sequences} \label{sect2}

\begin{thm} \label{th1a}
Let $\{Y_n, \mathscr{F}_n, n\ge 1\}$ be a martingale difference sequence with finite second moments, and fix any $c>0$.
Write  $M_n=\sum_{k=1}^n Y_k$ and ~$s_n^2=\mbox{Var}(M_n).
$
Suppose  $s_n\rightarrow \infty$  and the following conditions hold:
\begin{enumerate}[itemindent=2em]
 \item[$(i)$.]  $s_n^{-2}\sum_{k=1}^n Y_k^2 \rightarrow 1~a.s.$; 
 \item[$(ii)$.]  $\sum_{k=1}^{\infty} s_k^{-1} \mathbb{E}(|Y_k|I(|Y_k|>cs_k))<\infty$;
 \item[$(iii)$.]  $\sum_{k=1}^{\infty} s_k^{-4} \mathbb{E}(Y_k^4I(|Y_k|\le cs_k))<\infty$.
\end{enumerate}
Then 
there exists  a richer probability space with a standard Brownian motion $\{W(t), t\ge 0\}$ such that
\be
\frac{1}{\sqrt{s_n^2\log\log s_n}}|M_{n}-W(s^2_{n})|\to 0~~a.s., \label{asinva2}
\ee
and
\be
\frac{1}{s_n}\max\limits_{0\le t\le 1} \Big|M_{[nt]}-W(s^2_{[nt]})\Big|\stackrel{P}{\longrightarrow} 0.
\label{weakinva2}
\ee
\end{thm}

\begin{rem} \label{rem2.1}
For any $r\in [1,4]$, if $\mathbb{E}(|Y_n|^r)<\infty$ for all $n\ge 1$ and $\sum_{k=1}^{\infty} s_k^{-r} \mathbb{E}(|Y_k|^r)<\infty$, then it is clear that the conditions $(ii)$ and $(iii)$ hold.
\end{rem}

\begin{rem} \label{rem2.2a}
The proof of Theorem \ref{th1a} is based on the martingale version of the Skorohod embedding theorem (see Theorem A.1 in \cite{HH1980}).
For the strong invariance principle of  martingale difference sequences, Shao \cite{Shao1993} obtained that
if there exists a sequence $\{a_n, n\ge 1\}$ of non-decreasing positive numbers with $a_n\rightarrow \infty$ such that
\bestar
\sum_{k=1}^n \mathbb{E}(Y_k^2|\mathscr{F}_{k-1})-s_n^2=o(a_n)~~a.s.
\eestar
 and $\sum_{k=1}^{\infty} a_k^{-r/2} \mathbb{E}(|Y_k|^{r})<\infty$ for some $2\le r\le 4$,
then
\bestar
M(n)-W(s_n^2)=o\Big(\Big(a_n\Big(\log \frac{s_n^2}{a_n}+\log\log a_n\Big)\Big)^{1/2}\Big)~~a.s.
\eestar
This is exactly the same as (\ref{asinva2}) when $a_n=s_n^2$. By applying the  Skorohod embedding theorem, we can also obtain the
uniform weak convergence result (\ref{weakinva2}), which implies the functional central limit theorem (i.e., the weak invariance principle) for martingale difference sequences.
The corresponding result for i.i.d. random variables was provided in Theorem 2.1.2 of \cite{CR1981}.
\end{rem}

\begin{exam}  \label{th3a}
Let $\{Y_n, n\ge 1\}$ be  a stationary ergodic martingale difference sequence with $\mathbb{E}(Y_1^2)=1$, and write  $M_n=\sum_{k=1}^n Y_k$.
An application of Theorem \ref{th1a} gives that  there exists  a richer probability space with a standard Brownian motion $\{W(t), t\ge 0\}$ such that
\be
\frac{1}{\sqrt{n\log\log n}} |M_{n}-W(n)|\to 0~~a.s.,  \label{iden1}
\ee
and
\be
\frac{1}{\sqrt{n}}\sup\limits_{0\le t\le 1} \Big|M_{[nt]}-W([nt])\Big|\stackrel{P}{\longrightarrow} 0.   \label{iden2}
\ee
\end{exam}

\begin{proof}[Proofs of (\ref{iden1}) and (\ref{iden2})]
By the orthogonality of martingale differences, we have
\bestar
s_n^2:=\mbox{Var}(M_n)=\sum_{k=1}^n \mbox{Var}(Y_k)=\sum_{k=1}^n \mathbb{E}(Y_k^2)=n.
\eestar
Applying Birkhoff's ergodic theorem (see e.g. Theorem 10.6 in \cite{KA2002}) gives that
$
n^{-1}\sum_{k=1}^n Y_k^2 \rightarrow 1~a.s.
$
Fubini's theorem implies
 \bestar
\sum_{k=1}^{\infty} k^{-1/2} \mathbb{E}(|Y_k|I(|Y_k|> \sqrt{k}))&=&\sum_{k=1}^{\infty} k^{-1/2} \mathbb{E}(|Y_1|I(|Y_1|> \sqrt{k}))\\
&=&\mathbb{E}\Big(|Y_1|\sum_{k=1}^{\infty} k^{-1/2}I(Y_1^2>k)\Big)\le 2\mathbb{E}(Y_1^2)<\infty,
 \eestar
 and
 \bestar
 \sum_{k=1}^{\infty} k^{-2} \mathbb{E}(Y_k^4I(|Y_k|\le \sqrt{k}))&=&
 \mathbb{E}\Big(Y_1^4 \sum_{k=1}^{\infty} k^{-2}I(Y_1^2\le k)\Big)\\
 &=&
 \mathbb{E}\Big(Y_1^4 \sum_{k\ge Y_1^2} k^{-2}\Big)\le \mathbb{E}(Y_1^2)<\infty.
 \eestar
Now (\ref{iden1}) and (\ref{iden2}) follow from Theorem \ref{th1a}.
\end{proof}

In many applications, the object of interest $\{Z_n, n\ge 1\}$ is  not a martingale but a scaled martingale.
For example, $Z_n$ can be written as $Z_n=b_nM_n$, where  $\{M_n, n\ge 1\}$ is a zero-mean martingale and $\{b_n, n\ge 1\}$ is a regularly varying positive sequence.
Here $\{b_n, n\ge 1\}$ is said to vary  regularly (at $\infty$) with index $\alpha \in \mathbb{R}$ if
$b_{[nt]}/b_n \rightarrow t^{\alpha}$ as $n\rightarrow \infty$ for any $t>0$. If the conditions of Theorem \ref{th1a} are satisfied for the martingale difference sequence of $\{M_n, n\ge 1\}$,
then the strong invariance principle for $Z_n$ follows immediately from (\ref{asinva2}).
As for the weak invariance principle,  we can get it from (\ref{weakinva2}) for the process $\{Z_{[nt]}/b_{[nt]},~0\le t\le 1\}$.
By the following Theorem \ref{th3}, the weak invariance principle for  $\{Z_{[nt]}/b_n, ~0\le t\le 1\}$ can also be obtained.

\begin{thm} \label{th3}
Suppose that the conditions of Theorem \ref{th1a} are satisfied.
Let $\{s_n, n\ge 1\}$ and $\{b_n, n\ge 1\}$ be regularly varying positive sequences with indices $\alpha>0$ and $\beta\in \mathbb{R}$, respectively.
If $\alpha+\beta>0$, then there exists  a richer probability space with a standard Brownian motion $\{W(t), t\ge 0\}$ such that
\be
\frac{1}{s_n}\max\limits_{0\le t\le 1} \Big|\frac{b_{[nt]}M_{[nt]}}{b_n}-t^{\beta}W(s^2_{n}t^{2\alpha})\Big|\stackrel{P}{\longrightarrow} 0.
\label{weakinva}
\ee
\end{thm}

\section{Proofs of Theorems \ref{th1a} and \ref{th3}}  \label{sect3}

\begin{proof}[Proof of Theorem \ref{th1a}]  Let $\tilde{Y}_k=Y_kI(|Y_k|\le cs_k)$ for $k\ge 1$ and let $\tilde{M}_n=\sum_{k=1}^n (\tilde{Y}_k-\mathbb{E}(\tilde{Y}_k|\mathscr{F}_{k-1}))$ for $n\ge 1$.
 Since $\mathbb{E}(\tilde{Y}_k|\mathscr{F}_{k-1})=\mathbb{E}(Y_kI(|Y_k|> cs_k)|\mathscr{F}_{k-1})$ by noting that $\mathbb{E}(Y_k|\mathscr{F}_{k-1})=0$,
it follows from (ii) and Fubini's theorem that
\bestar
\mathbb{E}\Big(\sum_{k=1}^{\infty}  s_k^{-1}|\mathbb{E}(\tilde{Y}_k|\mathscr{F}_{k-1})|\Big)&\le& \mathbb{E}\Big(\sum_{k=1}^{\infty}  s_k^{-1}\mathbb{E}(|Y_k|I(|Y_k|> cs_k)|\mathscr{F}_{k-1})\Big)\\
&=& \sum_{k=1}^{\infty} \mathbb{E}( s_k^{-1}\mathbb{E}(|Y_k|I(|Y_k|> cs_k)|\mathscr{F}_{k-1}))\\
&=&
\sum_{k=1}^{\infty} s_k^{-1}\mathbb{E}(|Y_k|I(|Y_k|> cs_k))<\infty.
\eestar
This implies that
$
\sum_{k=1}^{\infty}  s_k^{-1}|\mathbb{E}(\tilde{Y}_k|\mathscr{F}_{k-1})|<\infty~a.s.
$
and hence by Kronecker's lemma,
\be
s_n^{-1}\sum_{k=1}^n |\mathbb{E}(\tilde{Y}_k|\mathscr{F}_{k-1})|\rightarrow 0~~a.s.  \label{snsum}
\ee
By (ii),
\bestar
\sum_{k=1}^{\infty}\mathbb{P}(\tilde{Y}_k\ne Y_k)=\sum_{k=1}^{\infty}  \mathbb{P}(|Y_k|>cs_k)\le \sum_{k=1}^{\infty} c^{-1}s_k^{-1} \mathbb{E}(|Y_k|I(|Y_k|>cs_k)) <\infty.
\eestar
Applying the Borel-Cantelli lemma gives that $\mathbb{P}(\tilde{Y}_k\ne Y_k~i.o.)=0$. Hence
\be
\frac{1}{s_n^2}\sum_{k=1}^n(Y_k^2-\tilde{Y}_k^2)\rightarrow 0~~a.s. ~~~~~~ \mbox{and}~~~~~~\frac{1}{s_n}\sup_{0\le t\le 1} |{M}_{[nt]}-\tilde{M}_{[nt]}|\rightarrow 0~~a.s. \label{eqinv1}
\ee

It follows from (iii) that
\bestar
\sum_{k=1}^{\infty}\mathbb{E}\Big( \frac{\tilde{Y}_k^2-\mathbb{E}(\tilde{Y}_k^2|\mathscr{F}_{k-1})}{s_k^2}\Big)^2\le \sum_{k=1}^{\infty} \frac{\mathbb{E}(\tilde{Y}_k^4)}{s_k^4}<\infty.
\eestar
Hence $\sum_{k=1}^{\infty} s_k^{-4}\mathbb{E}( (\tilde{Y}_k^2-\mathbb{E}(\tilde{Y}_k^2|\mathscr{F}_{k-1}))^2|\mathscr{F}_{k-1})<\infty~a.s.$
Applying Theorem 2.15 in \cite{HH1980} gives that
$
\sum_{k=1}^{\infty} s_k^{-2}({\tilde{Y}_k^2-\mathbb{E}(\tilde{Y}_k^2|\mathscr{F}_{k-1})})
$ converges a.s.
and hence by Kronecker's lemma,
\bestar
\frac{\sum_{k=1}^n (\tilde{Y}_k^2-\mathbb{E}(\tilde{Y}_k^2|\mathscr{F}_{k-1}))}{s_n^2}\rightarrow 0~~a.s.
\eestar
Combining this  with (\ref{eqinv1}) and the condition (i) shows
\be
\frac{\sum_{k=1}^n \mathbb{E}(\tilde{Y}_k^2|\mathscr{F}_{k-1})}{s_n^2}\rightarrow 1~~a.s.  \label{eqinv2}
\ee

By applying the martingale version of the Skorohod embedding theorem to $\tilde{M}_n$ (see, for instance,  Theorem A.1 in \cite{HH1980}),
there exists a richer probability space supporting a standard  Brownian motion $\{W(t), t\ge 0\}$ and  nonnegative random variables
 $\{\tau_n, n\ge 1\}$ with partial sums $T_n=\sum_{k=1}^n \tau_k$ such that for $n\ge 1$,  $\tau_n$ is $\mathscr{F}_n$-measurable, $\tilde{M}_n=W(T_n)$ and
\bestar
&&\mathbb{E}\big( \tau_n \big| \mathscr{F}_{n-1} \big) = \mathbb{E}\big( (\tilde{Y}_n')^2 \big| \mathscr{F}_{n-1} \big), \nonumber\\
&&\mathbb{E}\big( \tau_n^2 \big| \mathscr{F}_{n-1} \big) \le 4\mathbb{E}\big( (\tilde{Y}_n')^4 \big| \mathscr{F}_{n-1} \big),
\eestar
where $ \tilde{Y}_n'=\tilde{Y}_n-\mathbb{E}(\tilde{Y}_n|\mathscr{F}_{n-1})$ for $n\ge 1$.
Write $\tilde{\tau}_n=\tau_n-\mathbb{E}\big( \tau_n \big| \mathscr{F}_{n-1} \big)$ for $n\ge 1$, then $\{\tilde{\tau}_n, \mathscr{F}_n, n\ge 1\}$ is a martingale difference sequence  and
\bestar
\mathbb{E}\big( \tilde{\tau}_n^2 \big| \mathscr{F}_{n-1} \big) &=& \mathbb{E}\big(\tau_n^2 \big| \mathscr{F}_{n-1} \big) - \big(\mathbb{E}\big( \tau_n \big| \mathscr{F}_{n-1} \big)\big)^2
 \le \mathbb{E}\big( \tau_n^2 \big| \mathscr{F}_{n-1} \big)\\
 &\le&  4\mathbb{E}\big( (\tilde{Y}_n')^4 \big| \mathscr{F}_{n-1} \big)\le   64\mathbb{E}\big( \tilde{Y}_n^4 \big| \mathscr{F}_{n-1} \big).
\eestar
Hence  by using (iii), we have
\bestar
\sum_{k=1}^{\infty}\frac{\mathbb{E}\big( \tilde{\tau}_k^2 \big| \mathscr{F}_{k-1} \big)}{{s}_k^4} \le 64\sum_{k=1}^{\infty}\frac{\mathbb{E}\big( \tilde{Y}_k^4 \big| \mathscr{F}_{k-1} \big)}{s_k^4}<\infty~~a.s.
\eestar
Applying Theorem 2.15 in \cite{HH1980} gives that $\sum_{k=1}^{\infty} s_k^{-2}\tilde{\tau}_k$ converges a.s.  By Kronecker's lemma, we have
\bestar
\frac{T_n-\sum_{k=1}^n (\mathbb{E}\tilde{Y}_k^2|\mathscr{F}_{k-1})+\sum_{k=1}^n(\mathbb{E}(\tilde{Y}_k|\mathscr{F}_{k-1}))^2}{s_n^{2}}=\frac{\sum_{k=1}^n \tilde{\tau}_k}{s_n^2}\rightarrow 0~~a.s.
\eestar
 Hence it follows from (\ref{snsum}) and (\ref{eqinv2}) that $s_n^{-2} T_n\rightarrow 1~a.s.$
Noting that (see e.g. Lemma 14.7 in \cite{KA2002})
\be
\lim_{r\downarrow 1}\limsup_{t\rightarrow\infty} \sup_{t/r \le u\le rt}\frac{|W(u)-W(t)|}{\sqrt{t\log\log t}}=0~~a.s.,  \label{rate}
\ee
we get
\be
\tilde{M}_n-W(s_n^2)=W(T_n)-W(s_n^2)=o(\sqrt{s_n^2\log\log s_n})~~a.s. \label{hu1}
\ee
Now applying (\ref{eqinv1}) gives (\ref{asinva2}).

Write $\delta_n=\sup_{k\le n}|T_k-s_k^2|$, then by $s_n^{-2}T_n\rightarrow 1~a.s.$ we have  $s_n^{-2}\delta_n\rightarrow 0~a.s.$
Let
\be
\omega(f, t, h)=\sup_{0\le r, s\le t, |r-s|\le h}|f(r)-f(s)|,~~t, h>0  \label{modulus}
\ee
denote the modulus of continuity of the function $f: [0,\infty)\rightarrow \mathbb{R}$,
then  $\omega(W, 1+h, h)\rightarrow 0~a.s.$ as $h\rightarrow 0$ because of the almost sure continuity of Brownian paths.
For any $\varepsilon, h>0$,  applying the scaling property of $W$ gives
\bestar
&&\mathbb{P}\Big(s_n^{-1}\sup_{0\le t\le 1} \big|W(T_{[nt]})-W(s_{[nt]}^2) \big|>\varepsilon \Big)\\
&&~~~~~~\le  \mathbb{P}(\delta_n>hs_n^2)+\mathbb{P}(\omega(W, (1+h)s_n^2, hs_n^2)>\varepsilon s_n)\\
&&~~~~~~= \mathbb{P}(s_n^{-2}\delta_n>h)+\mathbb{P}(\omega(W, 1+h, h)>\varepsilon).
\eestar
Hence the right-hand side tends to $0$ as $n\rightarrow \infty$ and then $h\rightarrow 0$, and  (\ref{weakinva2}) follows by (\ref{eqinv1}).
\end{proof}

\begin{proof}[Proof of Theorem \ref{th3}]
For any fixed $0<b<1$,
Theorem 1.5.2 in \cite{Bingham1987} implies that as $n\rightarrow \infty$,
\bestar
\max_{b\le  t\le 1}  \Big|\frac{b_{[nt]}}{b_n}-t^{\beta}\Big|\rightarrow 0,~~~~\delta_n(b):=\max_{b\le  t\le 1}  \Big|\frac{s_{[nt]}^2}{s_n^2}-t^{2\alpha}\Big|\rightarrow 0,
\eestar
and hence
\be
I_{1,n}(b)&:=&\frac{1}{s_n}\max\limits_{b\le t\le 1} \Big|\frac{b_{[nt]}M_{[nt]}}{b_n}-t^{\beta}W(s^2_{[nt]})\Big|\nonumber\\
&\le &\frac{1}{s_n}\max\limits_{b\le t\le 1} \Big|\frac{b_{[nt]}(M_{[nt]}-W(s^2_{[nt]}))}{b_n}\Big|+\frac{1}{s_n}\max\limits_{b\le t\le 1} \Big|\Big(\frac{b_{[nt]}}{b_n}-t^{\beta}\Big)W(s^2_{[nt]})\Big|\nonumber\\
&\le& \frac{1}{s_n}\max_{0\le t\le 1}\Big|M_{[nt]}-W(s^2_{[nt]})\Big|\max\limits_{b< t\le 1}  \Big|\frac{b_{[nt]}}{b_n}\Big|
+\frac{1}{s_n}\max\limits_{0\le t\le s_n^2}|W(t)| \max_{b\le  t\le 1}  \Big|\frac{b_{[nt]}}{b_n}-t^{\beta}\Big| \nonumber\\
&\stackrel{P}{\rightarrow}& 0,  ~~~~n\rightarrow \infty,  \label{I1n}
\ee
where in the last step we have used Theorem \ref{th1a} and the scaling property of $\{W(t), t \ge 0\}$:
\bestar
s_n^{-1}\max\limits_{0\le t\le s_n^2}|W(t)| \stackrel{d}{=} \max_{0\le t\le 1} |W(t)|.
\eestar

By the scaling property and the almost sure continuity of sample paths of $\{W(t), t\ge 0\}$, we have that for fixed $0<b<1$,
\be
I_{2,n}(b)&:=&\frac{1}{s_n}\max\limits_{b\le t\le 1} \Big|t^{\beta}(W(s^2_{[nt]})-W(s^2_{n}t^{2\alpha}))\Big|\nonumber\\
&\le& \frac{\max\{b^{\beta}, 1\}}{s_n}\max\limits_{b\le t\le 1} \Big|W(s^2_{[nt]})-W(s^2_{n}t^{2\alpha})\Big|\nonumber\\
&\le& \frac{\max\{b^{\beta}, 1\}}{s_n} \omega(W, s_n^2, s_n^2\delta_n(b))\nonumber\\
&\stackrel{d}{=}& \max\{b^{\beta}, 1\} \omega(W, 1, \delta_n(b))\rightarrow 0~~a.s.,~~~~n\rightarrow \infty,  \label{I2n}
\ee
where $\omega(\cdot, \cdot, \cdot)$ is the modulus of continuity and defined as in (\ref{modulus}).

For any $0<\delta<\min\{\alpha, |\beta|, (\alpha+\beta)/2\}$,
by applying Theorem 1.5.6 in \cite{Bingham1987}, there exists $n_0$ such that
\bestar
b_{[nt]}/b_n\le 2t^{\beta-\delta}~~\mbox{and}  ~~(1/2)t^{\alpha+\delta}\le s_{[nt]}/s_n\le 2t^{\alpha-\delta}
\eestar
 hold
for  $n \ge 1$ and $0<t\le 1$ satisfying $nt\ge n_0$.
Define
\bestar
\delta_{m,n}=\sup_{m\le k\le n}|T_{k}/s_k^2-1|,
\eestar
where $\{T_k\}$ is defined as in the proof of Theorem \ref{th1a}. It follows from $T_n/s_n^2\rightarrow 1~a.s.$ that
\bestar
\limsup_{m\rightarrow\infty}\limsup_{n\rightarrow\infty}\mathbb{P}(\delta_{m,n}\ge 1/2)=0.
\eestar
If $\delta_{m,n}<1/2$ and $n_0 \le m\le bn$, then
\bestar
I_{3,m,n}(b)&:=&\frac{1}{s_n}\max\limits_{m/n \le t\le b} \Big|\frac{b_{[nt]}M_{[nt]}}{b_n}\Big|=\frac{1}{s_n}\max\limits_{m/n\le t\le b} \Big|\frac{b_{[nt]}W(T_{[nt]})}{b_n}\Big|\\
&\le&  \frac{2}{s_n}\max\limits_{m/n\le t\le b} t^{\beta-\delta} \max_{1/2\le u\le 2} |W(s_{[nt]}^2u)|\\
&\le& \frac{2}{s_n}\max\limits_{0\le t\le b} t^{\beta-\delta} \max_{t^{2\alpha+2\delta}/8\le u\le 8t^{2\alpha-2\delta}} |W(s_n^2u)|\\
&\stackrel{d}{=}& 2\max\limits_{0\le t\le b} t^{\beta-\delta} \max_{t^{2\alpha+2\delta}/8\le u\le 8t^{2\alpha-2\delta}} |W(u)|\\
&\le& 2\max_{0\le u\le 8b^{2\alpha-2\delta}} (u/8)^{r} |W(u)|,
\eestar
where $r=0$ when $\beta>0$ and $r=(\beta-\delta)/(2\alpha -2\delta)>-1/2$ when $\beta\le 0$.
Hence by applying the law of the iterated logarithm (see e.g. Theorem 13.18 in \cite{KA2002}):
\bestar
\limsup_{t\rightarrow 0} \frac{|W(t)|}{\sqrt{2t\log\log (1/t)}}=1~~~~a.s.,
\eestar
we get that
\be
&&\lim_{b\rightarrow 0}\limsup_{m\rightarrow\infty}\limsup_{n\rightarrow\infty}\mathbb{P}(I_{3,m,n}(b)>\varepsilon)
 \nonumber\\
&&~~~~~~~\le\limsup_{m\rightarrow\infty}\limsup_{n\rightarrow\infty}\mathbb{P}(\delta_{m,n}\ge 1/2)\nonumber\\
&&~~~~~~~~~~~~+\lim_{b\rightarrow 0}\limsup_{m\rightarrow\infty}\limsup_{n\rightarrow\infty}\mathbb{P}(I_{3,m,n}(b)>\varepsilon, \delta_{m,n}<1/2)
=0. \label{I3n}
\ee

Similarly,
\be
\lim_{b\rightarrow 0}\limsup_{n\rightarrow\infty}\mathbb{P}(I_{4,n}(b)>\varepsilon)=0,  \label{I4n}
\ee where
\bestar
I_{4,n}(b):=\frac{1}{s_n}\max\limits_{0\le t\le b} \Big|t^{\beta}W(s^2_{n}t^{2\alpha})\Big|\stackrel{d}{=}\max\limits_{0\le t\le b} |t^{\beta}W(t^{2\alpha})|.
\eestar
Observing that
\bestar
I_{5,m,n}:=\frac{1}{s_n}\max\limits_{0\le t\le m/n} \Big|\frac{b_{[nt]}M_{[nt]}}{b_n}\Big|\le \frac{1}{b_ns_n} \max_{1\le k\le m} |b_kM_k|,
\eestar
and  $b_ns_n\rightarrow \infty$ as $n\rightarrow \infty$ since $b_ns_n$ is a regularly varying sequence with index $\alpha+\beta >0$, we have
\be
\limsup_{m\rightarrow\infty}\limsup_{n\rightarrow\infty}\mathbb{P}(I_{5,m,n}>\varepsilon)=0. \label{I5n}
\ee
Now (\ref{weakinva}) follows from (\ref{I1n})$-$(\ref{I5n}).
\end{proof}

\section{Proofs of main results} \label{sect4}

Suppose that $\{t_n, n\ge 1\}$ is a sequence of positive numbers with $t_n\uparrow \infty$.
Let
\bestar
Z_n=X_n {I}(|X_n| \le t_n),~~~~~~ n\ge 1,
\eestar
  and let $\{\hat{Z}_n,\check{Z}_n,\hat{S}^*(n),\check{S}^*(n), n\ge 1\}$ be the corresponding variables by replacing $\{X_n, ~n\ge 1\}$ in the algorithms of $\{\hat{X}_n,\check{X}_n, \hat{S}(n), \check{S}(n), n\ge 1\}$ by $\{Z_n, n\ge 1\}$.
  Let $\mathscr{F}_0=\{\emptyset, \Omega\},~\mathscr{F}_1=\sigma(X_1)$ and
  $$ \mathscr{F}_n = \sigma(\varepsilon_{2},\cdots,\varepsilon_{n},U_2,\cdots,U_n,X_1,\cdots,X_n),~~~~n\ge 2. $$
For any $n \ge 1$, we have
\be
\mathbb{E}\big( \hat{Z}_{n+1} \big| \mathscr{F}_n \big) = p \frac{\hat{Z}_1 + \cdots + \hat{Z}_n}{n}+(1-p)\mathbb{E}(Z_{n+1})= \frac{p}{n}\hat{S}^*(n) + (1-p)\mathbb{E}(Z_{n+1}).
\label{YnFn0}
\ee
Define
$\hat{\gamma}_n = \frac{n+p}{n}$ for $n\ge 1$,~ $\hat{a}_1=1$  and
\be \label{defan}
\hat{a}_n = \prod_{k=1}^{n-1}\hat{\gamma}_k^{-1} = \frac{\Gamma(n)\Gamma(1+p)}{\Gamma(n+p)},~~~~n\ge 2.
\ee
Then
\be
\mathbb{E}\big( \hat{S}^*(n+1) \big| \mathscr{F}_n \big) = \hat{\gamma}_n\hat{S}^*(n) + (1-p)\mathbb{E}(Z_{n+1}) \label{YnFn1}
\ee
 and hence  $\{\hat{M}_n, \mathscr{F}_n, n\ge 1\}$
is a  martingale, where $\hat{M}_n=\hat{a}_n (\hat{S}^*(n)-\mathbb{E}(\hat{S}^*(n)))$.
Let $\{\hat{Y}_n, n\ge 1\}$ be the martingale differences, that is, $\hat{Y}_1=\hat{M}_1$ and $\hat{Y}_{n} = \hat{M}_{n}-\hat{M}_{n-1}$ for all $n\ge 2$.
It follows from \eqref{YnFn0} and \eqref{YnFn1} that for $n\ge 2$,
\be
\nonumber
\hat{Y}_{n} &=& \hat{a}_n \Big(\hat{S}^*(n)-\mathbb{E}(\hat{S}^*(n))- \hat{\gamma}_{n-1} (\hat{S}^*(n-1)-\mathbb{E}(\hat{S}^*(n-1))) \Big)
\\
&=& \hat{a}_n \Big(\hat{S}^*(n)-\hat{S}^*(n-1)-\frac{p}{n-1}\hat{S}^*(n-1)  - \big( \mathbb{E}(\hat{S}^*(n))-\hat{\gamma}_{n-1}\mathbb{E}(\hat{S}^*(n-1)) \big)\Big) \nonumber \\
\nonumber
&=& \hat{a}_n \Big( \hat{Z}_{n}-\frac{p}{n-1}\hat{S}^*(n-1)-(1-p)\mathbb{E}(Z_{n}) \Big)\\
\label{epsilonn+1}
&=& \hat{a}_n \Big( \hat{Z}_{n} - \mathbb{E}\big( \hat{Z}_{n} \big| \mathscr{F}_{n-1} \big)\Big).
\ee
For $n=1$, it is clear that we still have $\hat{Y}_{1}=\hat{a}_1 \big( \hat{Z}_{1} - \mathbb{E}\big( \hat{Z}_{1} \big| \mathscr{F}_{0} \big)\big)$.

Similarly, we can define $\check{\gamma}_n = \frac{n-p}{n}$  for $n\ge 1$,~ $\check{a}_1=1$  and
\be \label{checkan}
\check{a}_n = \prod_{k=1}^{n-1}\check{\gamma}_k^{-1} =\frac{\Gamma(n)\Gamma(1-p)}{\Gamma(n-p)},~~~~n\ge 2.
\ee
Let
$
\check{Y}_{n}= \check{a}_n \big( \check{Z}_{n} - \mathbb{E}\big( \check{Z}_{n} \big| \mathscr{F}_{n-1} \big)\big)$ for $n\ge 1$.
Then
$
\check{M}_n:=\sum_{k=1}^n \check{Y}_k=\check{a}_n (\check{S}^*(n)-\mathbb{E}(\check{S}^*(n)))
$
and
$\{\check{M}_n, \mathscr{F}_n, n\ge 1\}$
is a  martingale.

\begin{lemma} \label{lemmamoment} For any $m, n\in \mathbb{N}$, we have
\bestar
\mathbb{E}(\hat{Z}_n^{2m})=\mathbb{E}(\check{Z}_n^{2m}) \le \mathbb{E}(Z_n^{2m}).
\eestar
\end{lemma}

\begin{proof} It follows from the constructions of $\{\hat{Z}_n, n\ge 1\}$ and $\{\check{Z}_n, n\ge 1\}$ that $\mathbb{E}(\hat{Z}_n^{2m})=\mathbb{E}(\check{Z}_n^{2m})$.
We will prove  $\mathbb{E}(\hat{Z}_n^{2m}) \le \mathbb{E}(Z_n^{2m})$  by induction on $n$.
When $n=1$, it is clear that $\mathbb{E}(\hat{Z}_1^{2m}) = \mathbb{E}(Z_1^{2m})$. 
To prove the result for $n=k$, observe that
$$ \mathbb{E}\big( \hat{Z}_{k}^{2m} \big| \mathscr{F}_{k-1} \big) = (1-p) \mathbb{E}({Z}_{k}^{2m}) + p\frac{\sum_{j=1}^{k-1} \hat{Z}_j^{2m}}{k-1}.$$
Applying the
result for  $n\le k-1$ gives that
\bestar
 \mathbb{E}(\hat{Z}_{k}^{2m}) &=& (1-p) \mathbb{E}({Z}_{k}^{2m}) + p \frac{\sum_{j=1}^{k-1} \mathbb{E}(\hat{Z}_j^{2m})}{k-1} \\
&\le& (1-p) \mathbb{E}(Z_{k}^{2m}) + p \frac{\sum_{j=1}^{k-1} \mathbb{E}(Z_j^{2m})}{k-1} \le \mathbb{E}(Z_{k}^{2m}),
\eestar
and the proof is complete.
\end{proof}

\begin{lemma} \label{lemmalimita}
Suppose that $\{a_n, n\ge 1\}$ and $\{b_n, n\ge 1\}$ are  sequences of real numbers such that
\bestar
a_{n+1}=\frac{n+x}{n}a_n+b_{n+1},~~~~ n\ge 1.
\eestar
If $b_n\rightarrow b\in \mathbb{R}$, then we have
\be
a_n \sim \left\{
\begin{array}{ll}
bn/({1-x}), & x<1,\\
bn\log n, & x=1,\\
O(n^x), & x>1.
\end{array}\right.  \label{ansim}
\ee
Moreover, if $b_n=o(n^{-r})$ for some $r>0$,
then
\be
a_n=\left\{
\begin{array}{ll}
o(n^{1-r}), & x<1-r,\\
o(n^{x}\log n), & x=1-r,\\
 O(n^x), & x>1-r.
\end{array}
\right. \label{sumbgamma2}
\ee
\end{lemma}

\begin{proof}
Write
\bestar
\gamma_{j,n}(x)=\prod_{k=j}^{n-1} \frac{k+x}{k}=\frac{\Gamma(n+x)}{\Gamma(n)}\frac{\Gamma(j)}{\Gamma(j+x)},~~~~j\le n.
\eestar
Choose $n_0\in \mathbb{N}$ so that $n_0+x>0$,  then
\be
	a_{n} = a_{n_0}\gamma_{n_0,n}(x) + \sum_{j=n_0+1}^{n}b_j \gamma_{j,n}(x). \label{svar}
	\ee

Note that $\gamma_{j,n}(1)=n/j$ and for $x\ne 1$,
\bestar
\gamma_{j,n}(x)=
\frac{\Gamma(n+x)}{(x-1)\Gamma(n)}\Big(\frac{\Gamma(j)}{\Gamma(j+x-1)}-\frac{\Gamma(j+1)}{\Gamma(j+x)}\Big).
\eestar
Hence
\bestar
\sum_{j=n_0+1}^n \gamma_{j,n}(x)\sim \left\{
\begin{array}{ll}
{n}/({1-x}), & x<1,\\
n\log n, & x=1,\\
\frac{\Gamma(n_0+1)}{(x-1)\Gamma(n_0+x)} n^{x}, & x>1.
\end{array}\right.
\eestar
Now (\ref{ansim}) follows from Toeplitz's lemma.

Next, we assume $b_n=o(n^{-r})$.
Since ${\Gamma(n+x)}/{\Gamma(n)}\sim n^x$, there exists $c_1>1$ and $n_1\in \mathbb{N}$ such that
$n^x/c_1\le {\Gamma(n+x)}/{\Gamma(n)}\le c_1n^x$ for  $n\ge n_1$.
And for any $\varepsilon>0$, there exists $n_2\in \mathbb{N}$ such that $n_2\ge n_1$ and $|b_n| \le \varepsilon n^{-r}$ for $n\ge n_2$.
Hence
\be
\sum_{j=n_2}^{n} |b_j| \gamma_{j,n}(x) \le \varepsilon c_1^2 n^x \sum_{j=n_2}^{n}j^{-r-x}
\le \left\{
\begin{array}{ll}
C\varepsilon n^{1-r}, & r+x<1,\\
 C\varepsilon n^x\log n, & r+x=1,\\
 C\varepsilon n^x, & r+x>1,
\end{array}
\right. \label{sumbgamma}
\ee
where $C$ is a constant not depending on $n$ and $\varepsilon$.
(\ref{sumbgamma2}) follows from   (\ref{svar}) and (\ref{sumbgamma}) by noting that $\max_{1\le j\le n_2}\gamma_{j,n}(x)=O(n^x)$ holds for fixed $n_2$.
\end{proof}

For any $ n, j \in \mathbb{N} $ with $n\ge j$,  let
$$ N_j(n) = \# \{ l \le n : \hat{X}_l = X_j \} $$
denote the number of occurrences of the variable $X_j$ in the first $n$ steps of the algorithm (\ref{defhatX}).  Then we have
\be
\label{hats1}
\hat{S}(n)=\sum_{j=1}^{n} N_j(n) X_j,~~~~\hat{S}^*(n)=\sum_{j=1}^{n} N_j(n) Z_j .
\ee
When $N_j(n)\ge 1$ ( i.e., $n\ge j$ and $\varepsilon_j=0$),
we write $j=l_{j,1}<l_{j,2}<\cdots<l_{j, N_j(n)}\le n$ for the increasing sequence of steps of the algorithm (\ref{defhatX}) at which
$X_j$ appears. Construct the genealogical tree $\mathbb{T}_j(n)$ with vertex set $\{l_{j,1}, l_{j,2}, \cdots, l_{j, N_j(n)}\}$, rooted at $j=l_{j,1}$,
and such that for any $1\le a<b\le N_j(n)$, ~$(l_{j,a}, l_{j,b})$ is an edge of   $\mathbb{T}_j(n)$ if and only if $U_{l_{j,b}}=l_{j,a}$.
A vertex $l$ in  $\mathbb{T}_j(n)$ is called odd (respectively, even) when its distance to the root $j$ in  $\mathbb{T}_j(n)$ is an odd (respectively, even) number.
We write Odd$(\mathbb{T}_j(n))$ (respectively, Even$(\mathbb{T}_j(n))$) for the number of odd (respectively, even) vertices in $\mathbb{T}_j(n)$
 and set $ \Delta_j(n)= \text{Even}(\mathbb{T}_j(n)) - \text{Odd}(\mathbb{T}_j(n)) $.
If $N_j(n)=0$, then $\mathbb{T}_j(n)$ is the empty graph and
 we set $\Delta_j(n)=\text{Even}(\mathbb{T}_j(n)) =\text{Odd}(\mathbb{T}_j(n))=0$   by convention.

Observe that for $ n\ge j$, the variable $X_j$ appears exactly Even$(T_j(n))$ times and its opposite $-X_j$ appears exactly Odd$(T_j(n))$ times in $\check{S}(n)$. So
\be
\label{checks1}
\check{S}(n)=\sum_{j=1}^{n} \Delta_j(n) X_j,~~~~\check{S}^*(n)=\sum_{j=1}^{n} \Delta_j(n) Z_j .
\ee

\begin{lemma} \label{Njn} Let $j$ be a fixed integer. Then $N_j(n)/\omega_n\rightarrow 0~~a.s.$ and in $L_1$,
where
\bestar
\omega_n=\left\{
\begin{array}{ll}
\sqrt{n}, & ~~~~0<p<1/2,\\
\sqrt{n\log n}, & ~~~~p=1/2,\\
n, & ~~~~1/2<p<1.
\end{array}
\right.
\eestar
\end{lemma}

\begin{proof}
	For $ n \ge j $, $ N_j(n+1)-N_j(n)=1 $ if and only if $ \varepsilon_{n+1}=1 $ and $ U_{n+1}\in \mathbb{T}_j(n)$, otherwise $ N_j(n+1)-N_j(n)=0 $. This implies that 
	\be
	\mathbb{E}\big( N_j(n+1)  \big| \mathscr{F}_n \big)&=& N_j(n)+\mathbb{E}\big( N_j(n+1)-N_j(n) \big| \mathscr{F}_n \big) \nonumber\\
&=& N_j(n)+p\frac{N_j(n)}{n}=\frac{n+p}{n} N_j(n).  \label{enj}
	\ee
Hence
\be
\mathbb{E}(N_j(n))=\prod_{k=j}^{n-1} \frac{k+p}{k}  \mathbb{E}(N_j(j))=\frac{(1-p)\Gamma(j)\Gamma(n+p)}{\Gamma(n)\Gamma(j+p)}\sim \frac{(1-p)\Gamma(j)}{\Gamma(j+p)}n^{p},
\label{ENjn}
\ee
which implies that $\mathbb{E}(N_j(n))/\omega_n\rightarrow 0$. Since $ N_j(n) \ge 0 $, we have $N_j(n))/\omega_n\stackrel{L_1}{\rightarrow} 0$.

By simple calculations, for sufficiently large $n$ we have $\frac{(n+p)\omega_n}{n\omega_{n+1}}<1$  and then
\bestar
 \mathbb{E}\Big( \frac{N_j(n+1)}{\omega_{n+1}} \Big| \mathscr{F}_n \Big) = \frac{(n+p)\omega_n}{n\omega_{n+1}}\frac{N_j(n)}{\omega_n}\le \frac{N_j(n)}{\omega_n}.
 \eestar
 Hence there exists $n_0\in \mathbb{N}$ such that $\{{N_j(n)}/\omega_n, n\ge n_0\} $ is a supermartingale,
and the $L_1$ convergence implies the almost sure convergence.
\end{proof}

\begin{lemma} \label{deltajn} Let $j$ be a fixed integer. Then $\Delta_j(n)/\sqrt{n}\rightarrow 0~~a.s.$
\end{lemma}

\begin{proof} For $n\ge j$, we observe that $n+1$ is odd (respectively, even) in $\mathbb{T}_j(n+1)$ if and only if $n+1\in \mathbb{T}_j(n+1)$ and  $U_{n+1}$ is  even (respectively, odd) in $\mathbb{T}_j(n)$.
So
\bestar
\Delta_j(n+1)-\Delta_j(n)=\left\{
\begin{array}{cl}
1,&  \mbox{if}~ \varepsilon_{n+1}=1 \mbox{~and~} U_{n+1} \mbox{~is  odd in}~ \mathbb{T}_j(n),\\
-1, &  \mbox{if}~\varepsilon_{n+1}=1 \mbox{~and~} U_{n+1} \mbox{~is even in}~ \mathbb{T}_j(n),\\
0, & \mbox{otherwise}.
\end{array}
\right.
\eestar
This implies that
\be
	\mathbb{E}\big( \Delta_j(n+1)-\Delta_j(n) \big| \mathscr{F}_n \big)&=& -p\frac{\Delta_j(n)}{n}, \label{Edeltadiff1}\\
\mathbb{E}\big( (\Delta_j(n+1)-\Delta_j(n))^{2} \big| \mathscr{F}_n \big)&=& p\frac{N_j(n)}{n}.  \label{Edeltadiff2}
\ee
Write
$
 M_n=\check{a}_n (\Delta_j(n)-\mathbb{E}(\Delta_j(n)))
$ for $n\ge j$, where $\check{a}_n$ is defined as in (\ref{checkan}).  It follows from (\ref{Edeltadiff1}) that $\{M_n,\mathscr{F}_n, n\ge j\}$ is a martingale and
\be
\mathbb{E}(\Delta_j(n))=\prod_{k=j}^{n-1} \frac{k-p}{k}  \mathbb{E}(\Delta_j(j))=\frac{(1-p)\Gamma(j)\Gamma(n-p)}{\Gamma(n)\Gamma(j-p)}\sim \frac{(1-p)\Gamma(j)}{\Gamma(j-p)}n^{-p}.
\label{Edeltajn}
\ee
Applying (\ref{Edeltadiff2}) gives
\bestar
\mathbb{E}( \Delta_j^2(n+1))=\frac{n-2p}{n}\mathbb{E}(\Delta_j^2(n)) +  \frac{p}{n}\mathbb{E}(N_j(n)),
\eestar
and hence by (\ref{enj}),
\bestar
\mathbb{E}( \Delta_j^2(n+1))-\frac{\mathbb{E}(N_j(n+1))}{3}=\frac{n-2p}{n} \Big(\mathbb{E}( \Delta_j^2(n))-\frac{\mathbb{E}(N_j(n))}{3}\Big).
\eestar
Combining this with (\ref{ENjn})  shows that
\be
\mathbb{E}( \Delta_j^2(n))
=\frac{\mathbb{E}(N_j(n))}{3}+O(n^{-2p})\sim  \frac{(1-p)\Gamma(j)}{3\Gamma(j+p)}n^{p}.  \label{vardeltajn}
\ee

Let $Y_j=M_j$ and let $
Y_{n}=M_n-M_{n-1}=\check{a}_n( \Delta_j(n)-\check{\gamma}_{n-1}\Delta_j(n-1))
$
 for $n\ge j+1$. Then it follows from (\ref{ENjn}) and (\ref{vardeltajn}) that for $n\ge j$,
 \bestar
  \check{a}_n^{-2}\mathbb{E}(Y_n^2)&\le& 2\mathbb{E}(\Delta_j(n)-\Delta_j(n-1))^2 +\frac{2p^2\mathbb{E}(\Delta_j^2(n-1))}{(n-1)^2}\\
 &\le& \frac{2\mathbb{E}(N_j(n-1))}{n-1}+\frac{2\mathbb{E}(\Delta_j^2(n-1))}{(n-1)^2}=O(n^{p-1}).
 \eestar
Hence
\bestar
\sum_{n=j}^{\infty}\frac{\mathbb{E}(Y_n^2)}{n\check{a}_n^2}\le C\sum_{n=j}^{\infty}n^{p-2}<\infty.
\eestar
This implies  $\sum_{n=j}^{\infty} n^{-1}\check{a}_n^{-2}\mathbb{E}(Y_n^2|\mathscr{F}_{n-1})<\infty~a.s.$ and applying Theorem 2.15 in \cite{HH1980} gives that
$\sum_{n=j}^{\infty} n^{-1/2}\check{a}_n^{-1} Y_n$ converges a.s. Hence  by Kronecker's lemma,
\bestar
\frac{\Delta_j(n)-\mathbb{E}(\Delta_j(n))}{\sqrt{n}}=\frac{\sum_{k=j}^n Y_k}{\sqrt{n}\check{a}_n}\rightarrow 0~~a.s.
\eestar
The desired result follows from (\ref{Edeltajn}).
\end{proof}

\begin{lemma} \label{lemma7} Suppose that $\{Y_j, j\ge 1\}$ is a random variable sequence and $\{Y_{j,n}, j\ge 1, n\ge 1\}$ is a random variable array.
If $Y_{j,n}\rightarrow 0~a.s.$ for any fixed $j$ and $\mathbb{P}(Y_n\ne 0 ~i.o.)=0$, then
\bestar
\sum_{j=1}^n Y_{j,n}Y_j\rightarrow 0~~a.s.
\eestar
\end{lemma}

\begin{proof} Write
$
A_k=\cap_{m=k}^{\infty} \{Y_m=0\}
$ for $k\ge 1$.
Then on $A_k$, we have
\bestar
\lim_{n\rightarrow \infty}\sum_{j=1}^n Y_{j,n}Y_j= \lim_{n\rightarrow \infty}\sum_{j=1}^{k} Y_{j,n}Y_j=0~~a.s.
\eestar
Hence
$
\sum_{j=1}^n Y_{j,n}Y_j\rightarrow 0~~a.s.
$ on $\cup_{k=1}^{\infty} A_k=\{Y_n\ne 0~i.o.\}^c$ and the desired result follows.
\end{proof}

\begin{lemma} \label{lemma5} If $\mathbb{E}(X_1^2)<\infty$, then we have
\be
\mbox{Var}(\check{S}^*(n)) \sim \frac{n\check{\sigma}^2}{2p+1} \label{varcheck}
\ee
and
\be
&&\mbox{Var}(\hat{S}^*(n)) \sim \left\{
\begin{array}{ll}
\sigma^2 n/(1-2p), & p<1/2,\\
\sigma^2n\log n, & p=1/2,\\
O(n^{2p}), & p>1/2.
\end{array}\right.  \label{evarhat2}
\ee
If $\mathbb{E}(X_1^2)<\infty$ and $t_n=\sqrt{n}$, then
$\mathbb{E}(\check{S}^*(n)) =(1-p)m_1n/(1+p)+o(\sqrt{n})$
and
\be
&&\mathbb{E}(\hat{S}^*(n))=m_1n+\left\{
\begin{array}{ll}
o(\sqrt{n}), & p<1/2,\\
O(\sqrt{n}), & p=1/2,\\
O(n^{p}), & p>1/2.
\end{array}\right.  \label{evarhat1}
\ee
\end{lemma}

\begin{proof}  We only prove  (\ref{evarhat2}) and (\ref{evarhat1})  since the proofs for $\mathbb{E}(\check{S}^*(n))$ and $\mbox{Var}(\check{S}^*(n))$ are similar and we omit the details.
It follows from (\ref{YnFn1}) that
\be
\mathbb{E}( \hat{S}^*(n+1))=\frac{n+p}{n}\mathbb{E}(\hat{S}^*(n)) + (1-p)\mathbb{E}(Z_{n+1}). \label{indehat}
\ee
Applying Lemma \ref{lemmalimita} gives $\mathbb{E}( \hat{S}^*(n))/n\rightarrow m_1$ since $\mathbb{E}(Z_{n+1})\rightarrow m_1$.
Observe that
	\be
	\label{EFnshatstar2}
	\mathbb{E}\big( \hat{Z}_{n+1}^2 \big| \mathscr{F}_n \big)
	= \frac{p}{n}\sum_{k=1}^n \hat{Z}_{k}^2 + (1-p)\mathbb{E}(Z_{n+1}^2)
    =\frac{p}{n}\sum_{k=1}^n N_k(n){Z}_{k}^2 + (1-p)\mathbb{E}(Z_{n+1}^2).
	\ee
Hence $\mathbb{E}( \hat{Z}_{n+1}^2)= pn^{-1}\sum_{k=1}^n \mathbb{E}(N_k(n))\mathbb{E}({Z}_{k}^2) + (1-p)\mathbb{E}(Z_{n+1}^2)$.
Combining this with   Lemma \ref{Njn},  Toeplitz's lemma and the fact that $\sum_{k=1}^n \mathbb{E}(N_k(n))=n$ implies that
$
\mathbb{E}( \hat{Z}_{n}^2) \rightarrow m_2.
$
By using \eqref{YnFn0} and \eqref{EFnshatstar2}, we have
	\bestar
	\text{Var}\big( \hat{S}^*(n+1) \big| \mathscr{F}_n \big)
	&=& \text{Var}\big( \hat{Z}_{n+1}\big| \mathscr{F}_n \big)=  \mathbb{E}\big( \hat{Z}_{n+1}^2 \big| \mathscr{F}_n \big)  - \big( \mathbb{E}\big( \hat{Z}_{n+1} \big| \mathscr{F}_n \big) \big)^2 \\
	&=& \frac{p}{n}\sum_{k=1}^n \hat{Z}_{k}^2 + (1-p)\mathbb{E}(Z_{n+1}^2)- \Big( \frac{p}{n}\hat{S}^*(n) + (1-p)\mathbb{E}(Z_{n+1}) \Big)^2
	\eestar
and hence
\bestar
	\mathbb{E}(\text{Var}\big( \hat{S}^*(n+1) \big| \mathscr{F}_n \big))
	&=& \frac{p}{n}\sum_{k=1}^n \mathbb{E}(\hat{Z}_{k}^2) + (1-p)\mathbb{E}(Z_{n+1}^2)- \mathbb{E}\Big( \frac{p}{n}\hat{S}^*(n) + (1-p)\mathbb{E}(Z_{n+1}) \Big)^2\\
   &=& b_{n+1}- \frac{p^2}{n^2}\mbox{Var}(\hat{S}^*(n)),
	\eestar
where
\bestar
b_{n+1}:=\frac{p}{n}\sum_{k=1}^n \mathbb{E}(\hat{Z}_{k}^2) + (1-p)\mathbb{E}(Z_{n+1}^2)-\Big( \frac{p}{n}\mathbb{E}(\hat{S}^*(n)) + (1-p)\mathbb{E}(Z_{n+1}) \Big)^2\rightarrow \sigma^2.
\eestar
Then applying  \eqref{YnFn1} gives that
	\bestar
	\text{Var}\big( \hat{S}^*(n+1) \big)
	&=& \mathbb{E} \big( \text{Var}\big( \hat{S}^*(n+1) \big| \mathscr{F}_n \big) \big) + \text{Var} \big( \mathbb{E}\big( \hat{S}^*(n+1) \big| \mathscr{F}_n \big) \big) \\
&=&  b_{n+1}- \frac{p^2}{n^2}\mbox{Var}(\hat{S}^*(n))
	+ \Big(1+\frac{p}{n}\Big)^2\text{Var}\big( \hat{S}^*(n) \big) \\
&=&  \frac{n+2p}{n}\mbox{Var}(\hat{S}^*(n))+ b_{n+1}.
	\eestar
Now (\ref{evarhat2}) follows from (\ref{ansim}).		

Next,  assume that $t_n=\sqrt{n}$. It follows from (\ref{indehat}) that
\bestar
\mathbb{E}( \hat{S}^*(n+1))-m_1(n+1)=\frac{n+p}{n}\Big(\mathbb{E}(\hat{S}^*(n))-m_1n\Big) - (1-p)\mathbb{E}(X_{n+1}I(|X_{n+1}|>\sqrt{n+1})).
\eestar
Define
\bestar
\tilde{a}_n=\mathbb{E}(\hat{S}^*(n))-m_1n, ~~~~\tilde{b}_n=-(1-p)\mathbb{E}(X_{n}I(|X_{n}|>\sqrt{n})).
\eestar
Then
\bestar
|\tilde{b}_n|\le n^{-1/2}\mathbb{E}(X_1^2I(|X_1|>\sqrt{n}))=o(n^{-1/2}),
\eestar
and hence (\ref{evarhat1}) follows from Lemma \ref{lemmalimita}  for $p\ne 1/2$.

If $p=1/2$, then by Fubini's theorem,
\bestar
\sum_{j=n_2}^{n} |\tilde{b}_j| \gamma_{j,n}(1/2) &\le& c_1^2 \sqrt{n} \sum_{j=n_2}^{n}j^{-1/2}\mathbb{E}(|X_{1}|I(|X_{1}|>\sqrt{j}))\\
&\le& c_1^2 \sqrt{n} \mathbb{E}\Big(|X_{1}|\sum_{j=1}^n j^{-1/2}I(|X_{1}|>\sqrt{j})\Big)\\
&\le&  c_1^2 \sqrt{n} \mathbb{E}\Big(|X_{1}|\sum_{X_1^2>j} j^{-1/2})\Big)\le (c_1^2/2) \sqrt{n}\mathbb{E}(X_1^2)=O(\sqrt{n}),
\eestar
where $n_2,~c_1$ and $\gamma_{j,n}(\cdot)$ are defined as in the proof of Lemma \ref{lemmalimita}. Hence we have $\tilde{a}_n=O(\sqrt{n})$ from (\ref{svar})
and the fact that $\max_{1\le j\le n_2} \gamma_{j,n}(1/2)=O(\sqrt{n})$. The proof of Lemma \ref{lemma5} is complete.
\end{proof}

\begin{lemma} \label{lemma6}
Let $\{c_n, n\ge 1\}$ and $\{d_n, n\ge 1\}$ be sequences of positive numbers such that $c_n\rightarrow 1$ and $d_n\rightarrow 1$.
If (\ref{asinva2}) and (\ref{weakinva2}) hold, then
\be
\frac{1}{\sqrt{s_n^2\log\log s_n}}\max\limits_{0\le t\le 1} \Big|c_{[nt]}M_{[nt]}-W(d_{[nt]}s^2_{[nt]})\Big|\to 0~~a.s. \label{asinva5}
\ee
and
\be
\frac{1}{s_n}\max\limits_{0\le t\le 1} \Big|c_{[nt]}M_{[nt]}-W(d_{[nt]}s^2_{[nt]})\Big|\stackrel{P}{\longrightarrow} 0.
\label{weakinva5}
\ee
If (\ref{weakinva}) holds, then
\be
\frac{1}{s_n}\max\limits_{0\le t\le 1} \Big|\frac{c_{[nt]}b_{[nt]}M_{[nt]}}{b_n}-t^{\beta}W(d_ns^2_{n}t^{2\alpha})\Big|\stackrel{P}{\longrightarrow} 0. \label{pinva5}
\ee
Moreover, (\ref{asinva5})$-$(\ref{pinva5}) still hold by replacing $c_{[nt]}$ and $d_{[nt]}$ by $c_n$ and $d_n$ respectively.
\end{lemma}

\begin{proof} We only prove (\ref{asinva5}) since the others are similar and we omit the details.
If (\ref{asinva2}) holds, then  applying the law of the iterated logarithm for $\{W(t), t\ge 0\}$ gives
\bestar
\limsup_{n\rightarrow \infty}\frac{\max\limits_{1\le k\le n}|M_{k}|}{\sqrt{2s_n^2\log\log s_n}}=\limsup_{n\rightarrow \infty}\frac{\max\limits_{1\le k\le n}|W(s_k^2)|}{\sqrt{2s_n^2\log\log s_n}}\le
\limsup_{n\rightarrow \infty}\frac{\max\limits_{0\le t\le s_n^2}|W(t)|}{\sqrt{2s_n^2\log\log s_n}}=1~~~~a.s.
\eestar
Hence
\be
&&\frac{1}{\sqrt{s_n^2\log\log s_n}}\max\limits_{0\le t\le 1}|c_{[nt]}-1||M_{[nt]}|\nonumber\\
&&~~~~\le \max\limits_{m\le k\le n}|c_{k}-1|\frac{\max\limits_{1\le k\le n}|M_{k}|}{\sqrt{s_n^2\log\log s_n}}+\frac{\max\limits_{1\le k\le m} |c_{k}-1||M_k|}{\sqrt{s_n^2\log\log s_n}}\rightarrow 0~~a.s.
\label{adddiff1}
\ee
as $n\rightarrow \infty$ and then $m\rightarrow \infty$.
Define $\Delta_{m}=\max_{k\ge m}|d_k-1|$. Then $\Delta_m\rightarrow 0$ as $m\rightarrow \infty$ and hence by (\ref{rate}),
\bestar
&&\frac{1}{\sqrt{s_n^2\log\log s_n}}\max\limits_{0\le t\le 1} \Big|W(s^2_{[nt]})-W(d_{[nt]}s^2_{[nt]})\Big|\\
&&~~\le \frac{1}{\sqrt{s_n^2\log\log s_n}}\max\limits_{1\le k\le m} \Big|W(s^2_{k})-W(d_{k}s^2_{k})\Big|
+\sup_{u, v\le s_n^2, |u-v|\le \Delta_{m}v} \frac{|W(u)-W(v)|}{\sqrt{s_n^2\log\log s_n}}\rightarrow 0~~a.s.
\eestar
as $n\rightarrow \infty$ and then $m\rightarrow \infty$. Combining this with  (\ref{asinva2}) and (\ref{adddiff1}) implies  (\ref{asinva5}).
\end{proof}

\begin{proof}[Proof of Theorem \ref{LLN}]
 Choose $t_n=n$ for all $n\ge 1$.
By Fubini's theorem, we have
\bestar
\sum_{n=1}^{\infty}\mathbb{P}(Z_n\ne X_n)&=&\sum_{n=1}^{\infty}\mathbb{P}(|X_n|>n)=\sum_{n=1}^{\infty}\mathbb{E}(I(|X_1|>n))\\
&=&\mathbb{E}\Big(\sum_{n=1}^{\infty}I(|X_1|>n)\Big)\le \mathbb{E}(|X_1|)<\infty,
\eestar
and hence $\mathbb{P}(Z_n\ne X_n~i.o.)=0$ follows from the Borel-Cantelli Lemma.  By (\ref{hats1}), (\ref{checks1}) and Lemmas \ref{Njn}$-$\ref{lemma7},  we have
\bestar
\frac{1}{n} (\hat{S}^*(n)-\hat{S}(n))=\sum_{j=1}^n \frac{N_j(n)}{n} (Z_j-X_j)\rightarrow 0~~a.s.
\eestar
and
\bestar
\frac{1}{n} (\check{S}^*(n)-\check{S}(n))= \sum_{j=1}^n \frac{\Delta_j(n)}{n} (Z_j-X_j)\rightarrow 0~~a.s.
\eestar
Applying  Lemma \ref{lemmamoment} gives that
\bestar
\mathbb{E}(\hat{Y}_{n}^2)= \hat{a}_n^2 \Big(\mathbb{E}(\hat{Z}_{n}^2)-\mathbb{E}\big(\big( \mathbb{E}\big( \hat{Z}_{n} \big| \mathscr{F}_{n-1} \big)\big)^2\big)\Big)\le \hat{a}_n^2 \mathbb{E}(\hat{Z}_{n}^2)\le \hat{a}_n^2 \mathbb{E}({Z}_{n}^2).
\eestar
Hence by  Fubini's theorem, we have
\bestar
\sum_{n=1}^{\infty}\frac{\mathbb{E}(\hat{Y}_{n}^2)}{n^2\hat{a}_n^2}&\le& \sum_{n=1}^{\infty}\frac{\mathbb{E}(Z_{n}^2)}{n^2}=\sum_{n=1}^{\infty}\frac{\mathbb{E}(X_1^2I(|X_1|\le n))}{n^2}\\
&=& \mathbb{E}\Big(\sum_{n=1}^{\infty}\frac{X_1^2I(|X_1|\le n)}{n^2}\Big)=\mathbb{E}\Big(X_1^2\sum_{n\ge |X_1|}n^{-2}\Big)\le \mathbb{E}(|X_1|)<\infty.
\eestar
This implies $\sum_{n=1}^{\infty} n^{-2}\hat{a}_n^{-2}\mathbb{E}(\hat{Y}_n^2|\mathscr{F}_{n-1})<\infty~a.s.$ and
applying Theorem 2.15 in \cite{HH1980} yields that $\sum_{n=1}^{\infty} n^{-1}\hat{a}_n^{-1}\hat{Y}_n$ converges a.s. By Kronecker's lemma, we have
\bestar
\frac{\hat{S}^*(n)-\mathbb{E}(\hat{S}^*(n))}{n}=\frac{\sum_{k=1}^n \hat{Y}_k}{n\hat{a}_n}\rightarrow 0~~a.s.
\eestar
Similarly,
\bestar
\frac{\check{S}^*(n)-\mathbb{E}(\check{S}^*(n))}{n}\rightarrow 0~~a.s.
\eestar
The same arguments as we used in the proof of Lemma \ref{lemma5} show that if $E|X_1|<\infty$ then we have $n^{-1}\mathbb{E}(\hat{S}^*(n))\rightarrow m_1$ and
$
n^{-1}\mathbb{E}( \check{S}^*(n))\rightarrow {(1-p)m_1}/(1+p).
$
Now Theorem \ref{LLN} follows by combining the above facts.
\end{proof}

\begin{proof}[Proof of Theorem \ref{th4.1}] Choose $t_n=\sqrt{n}$ for all $n\ge 1$.
By Fubini's theorem, we have
\bestar
\sum_{n=1}^{\infty}\mathbb{P}(Z_n\ne X_n)&=&\sum_{n=1}^{\infty}\mathbb{P}(|X_n|>\sqrt{n})=\sum_{n=1}^{\infty}\mathbb{E}(I(X_1^2>n))\\
&=&\mathbb{E}\Big(\sum_{n=1}^{\infty}I(X_1^2>n)\Big)\le \mathbb{E}(X_1^2)<\infty,
\eestar
and hence $\mathbb{P}(Z_n\ne X_n~i.o.)=0$ by the Borel-Cantelli Lemma. By (\ref{hats1}) and Lemmas \ref{Njn}  and \ref{lemma7}, we have
\be
\frac{1}{\omega_n} (\hat{S}^*(n)-\hat{S}(n))= \sum_{j=1}^n \frac{N_j(n)}{\omega_n} (Z_j-X_j)\rightarrow 0~~a.s.  \label{diffe1}
\ee
for $p\le 1/2$, where $\omega_n$ is defined as in Lemma \ref{Njn}.
It follows from Lemma \ref{lemma5} that $\mathbb{E}(\hat{S}^*(n))=m_1n+o(\omega_n)$. This implies that  for $p<1/2$,
\be
\frac{\max\limits_{0\le t\le 1}\Big|\mathbb{E}(\hat{S}^*([nt]))-m_1[nt]\Big|}{\sqrt{n}}\rightarrow 0,  \label{esdiff}
\ee
 and for $p=1/2$,
\be
\frac{\max\limits_{1/n\le t\le 1}\Big|\frac{\mathbb{E}(\hat{S}^*([nt]))-m_1[nt]}{\sqrt{[nt]}}\Big|}{\sqrt{\log n}}\rightarrow 0.  \label{esdiff1}
\ee

Applying  Lemma \ref{lemmamoment} gives
\bestar
\mathbb{E}(\hat{Y}_{n}^4)\le 8\hat{a}_n^4 \Big(\mathbb{E}(\hat{Z}_{n}^4)+\mathbb{E}\big(\big( \mathbb{E}\big( \hat{Z}_{n} \big| \mathscr{F}_{n-1} \big)\big)^4\big)\Big)\le 16\hat{a}_n^4 \mathbb{E}(\hat{Z}_{n}^4)\le 16\hat{a}_n^4 \mathbb{E}(Z_{n}^4).
\eestar
Write $\hat{s}_n^2=\mbox{Var}(\hat{M}_n)=\hat{a}_n^2 \mbox{Var}(\hat{S}^*(n))$. Then by (\ref{evarhat2}),
\be
\sum_{n=1}^{\infty}\frac{\mathbb{E}(\hat{Y}_{n}^4)}{\hat{s}_n^4}&\le& 16\sum_{n=1}^{\infty}\frac{\mathbb{E}(Z_{n}^4)}{(\mbox{Var}(\hat{S}^*(n)))^2}\le C\sum_{n=1}^{\infty}\frac{\mathbb{E}(Z_{n}^4)}{n^2}\nonumber\\
&=&C\sum_{n=1}^{\infty}\frac{\mathbb{E}(X_1^4I(|X_1|\le \sqrt{n}))}{n^2}
= C\mathbb{E}\Big(\sum_{n=1}^{\infty}\frac{X_1^4I(|X_1|\le \sqrt{n})}{n^2}\Big)\nonumber\\
&=&C\mathbb{E}\Big(X_1^4\sum_{n\ge X_1^2}n^{-2}\Big)\le C\mathbb{E}(X_1^2)<\infty.  \label{EYn4}
\ee
This implies
\bestar
\sum_{n=1}^{\infty} \frac{\mathbb{E}((\hat{Y}_{n}^2-\mathbb{E}(Y_n^2|\mathscr{F}_{n-1}))^2|\mathscr{F}_{n-1})}{\hat{s}_n^4}\le
\sum_{n=1}^{\infty} \frac{\mathbb{E}(\hat{Y}_{n}^4|\mathscr{F}_{n-1})}{\hat{s}_n^4}<\infty~~a.s.
\eestar
Applying Theorem 2.15 in \cite{HH1980} yields that $\sum_{n=1}^{\infty} \hat{s}_n^{-2}(\hat{Y}_{n}^2-\mathbb{E}(Y_n^2|\mathscr{F}_{n-1}))$ converges a.s.
and hence by Kronecker's lemma,
\be
\frac{\sum_{k=1}^n (\hat{Y}_{k}^2- \mathbb{E}(\hat{Y}_{k}^2|\mathscr{F}_{k-1}))}{\hat{s}_n^2} \rightarrow 0~~a.s.   \label{Eyk2}
\ee

Using the same arguments as in the proof of Theorem \ref{LLN} yields $
n^{-1}\sum_{k=1}^n \hat{Z}_k^i\rightarrow m_i~~a.s.
$  for $i=1,2$.
Hence
\bestar
\mathbb{E}(\hat{Z}_{n+1}^i|\mathscr{F}_{n})=(1-p) \mathbb{E}(Z_{n+1}^i)+p\frac{\sum_{k=1}^n \hat{Z}_{k}^i}{n}\rightarrow m_i~~a.s.,~~~~~~i=1,2.
\eestar
This implies that
\bestar
\hat{a}_{n+1}^{-2}\mathbb{E}(\hat{Y}_{n+1}^2|\mathscr{F}_{n})=\mathbb{E}(\hat{Z}_{n+1}^2|\mathscr{F}_{n})-(\mathbb{E}(\hat{Z}_{n+1}|\mathscr{F}_{n}))^2\rightarrow \sigma^2~~a.s.
\eestar
By applying (\ref{evarhat2}), Toeplitz's Lemma   and the fact that $\hat{a}_n\sim n^{-p}\Gamma(1+p)$, we have that for $p\le 1/2$,
\bestar
\frac{\sum_{k=1}^n \mathbb{E}(\hat{Y}_{k}^2|\mathscr{F}_{k-1})}{\hat{s}_n^2}= \frac{\sum_{k=1}^n \hat{a}_{k}^2 (\hat{a}_{k}^{-2}\mathbb{E}(\hat{Y}_{k}^2|\mathscr{F}_{k-1}))}{\sum_{k=1}^n \hat{a}_{k}^2}
 \frac{\sum_{k=1}^n \hat{a}_{k}^2}{\hat{a}_n^2\mbox{Var}(\hat{S}^*(n)) } \rightarrow 1~~a.s.
\eestar
Combining this with (\ref{Eyk2}) yields $\hat{s}_n^{-2}\sum_{k=1}^n \hat{Y}_{k}^2\rightarrow 1~a.s.$
Now using (\ref{EYn4}) and applying  Theorem \ref{th1a}  and Remark \ref{rem2.1} gives that  there exists a richer probability space with a standard Brownian motion $\{\tilde{W}(t), t\ge 0\}$ satisfying that
\be
\frac{1}{\sqrt{\hat{s}_n^2\log\log \hat{s}_n}}\Big|\hat{a}_n(\hat{S}^*(n)-\mathbb{E}(\hat{S}^*(n)))- \tilde{W}(\hat{s}_n^2)\Big|\rightarrow  0~~~a.s.,  \label{asconvhat1}
\ee
and
\be
\frac{1}{\hat{s}_n}\sup_{0\le t\le 1}\Big|\hat{a}_{[nt]}(\hat{S}^*([nt])-\mathbb{E}(\hat{S}^*([nt])))- \tilde{W}(\hat{s}_{[nt]}^2)\Big|\stackrel{P}{\rightarrow} 0.  \label{probconvhat1}
\ee
Since $\hat{a}_n\sim n^{-p}\Gamma(1+p)$,
it follows from (\ref{evarhat2}) that
\be
\hat{s}_n^2= \hat{a}_n^2 \mbox{Var}(\hat{S}^*(n)) \sim \left\{
\begin{array}{ll}
\sigma^2 (\Gamma(1+p))^2 n^{1-2p}/(1-2p), & p<1/2,\\
\sigma^2(\Gamma(3/2))^2 \log n, & p=1/2.
\end{array}\right.  \label{snsim}
\ee

If $p<1/2$,
then by applying Theorem \ref{th3} with $b_n=\hat{a}_n^{-1}$,
\bestar
\frac{1}{\hat{s}_n}\Big|\hat{a}_n(\hat{S}^*([nt])-\mathbb{E}(\hat{S}^*([nt])))-t^{p}\tilde{W}(\hat{s}_n^2 t^{1-2p})\Big|\stackrel{P}{\rightarrow} 0.
\eestar
By applying Lemma \ref{lemma6} with
$c_n=\hat{a}_n^{-1}n^{-p}\Gamma(1+p)$ and $d_n=\sigma^2 (\Gamma(1+p))^2 \hat{s}_n^{-2} n^{1-2p}/(1-2p)$, we have
\be
\frac{1}{\sqrt{n^{1-2p}\log\log n}} \Big|\frac{\sqrt{1-2p}(\hat{S}^*(n)-\mathbb{E}(\hat{S}^*(n)))}{\sigma n^{p}}-\frac{\tilde{W}(\hat{b}_0^2 n^{1-2p})}{\hat{b}_0}\Big|\rightarrow 0~~~a.s.,
\label{th4.1eq3}
\ee
and
\be
\frac{1}{\sqrt{n^{1-2p}}} \max_{0\le t\le 1}\Big|\frac{\sqrt{1-2p}(\hat{S}^*([nt])-\mathbb{E}(\hat{S}^*([nt])))}{\sigma n^{p}}-\frac{t^{p}\tilde{W}(\hat{b}_0^2 (nt)^{1-2p})}{\hat{b}_0}\Big|\stackrel{P}{\rightarrow} 0.
\label{th4.1eq4}
\ee
where $\hat{b}_0=\sigma \Gamma(1+p)/\sqrt{1-2p}$.  Define $W(t)=\tilde{W}(\hat{b}_0^2t)/\hat{b}_0$.
Then $\{W(t), t\ge 0\}$ is a Brownian motion and  (\ref{th4.1eq1})$-$(\ref{th4.1eq2}) follows from
(\ref{diffe1}), (\ref{esdiff}), (\ref{th4.1eq3}) and (\ref{th4.1eq4}).

If $p=1/2$, then by applying Lemma \ref{lemma6} with
$c_n=\hat{a}_n^{-1}n^{-1/2}\Gamma(3/2)$ and $d_n=\sigma^2 (\Gamma(3/2))^2 \hat{s}_n^{-2} \log n$,
it follows from (\ref{asconvhat1}) and (\ref{probconvhat1}) that
\be
\frac{1}{\sqrt{\log n\log\log \log n}} \Big|\frac{\hat{S}^*(n)-\mathbb{E}(\hat{S}^*(n))}{\sigma \sqrt{n}}-\frac{\tilde{W}(\hat{b}_1^2\log n)}{\hat{b}_1}\Big|\rightarrow 0~~~a.s.,
\label{th4.1eq5}
\ee
and
\be
\frac{1}{\sqrt{\log n}} \max_{1/n \le t\le 1}\Big|\frac{\hat{S}^*([nt])-\mathbb{E}(\hat{S}^*([nt]))}{\sigma \sqrt{[nt]}}-\frac{\tilde{W}(\hat{b}_1^2 \log ([nt]))}{\hat{b}_1}\Big| \stackrel{P}{\rightarrow} 0.
\label{th4.1eq6}
\ee
where $\hat{b}_1=\sigma \Gamma(3/2)$.
By setting  $W(t)=\tilde{W}(\hat{b}_1^2t)/\hat{b}_1$  and using (\ref{diffe1}), (\ref{esdiff1}), (\ref{th4.1eq5}) and (\ref{th4.1eq6}), we can get (\ref{th4.1eq7}) and
\be
\frac{1}{\sqrt{\log n}} \max_{1/n \le t\le 1}\Big|\frac{\hat{S}([nt])-m_1[nt]}{\sigma \sqrt{[nt]}}-{W}(\log ([nt]))\Big| \stackrel{P}{\rightarrow} 0.
\label{th4.1eq9}
\ee
Observe that
\bestar
&&\frac{1}{\sqrt{\log n}} \max_{1/n\le  t\le 1}\Big|{W}(\log (nt))-{W}(\log ([nt]))\Big| \\
&&~~~~\stackrel{d}{=}\max_{1/n\le  t\le 1}\Big|{W}\Big(\frac{\log (nt)}{\log n}\Big)-{W}\Big(\frac{\log ([nt])}{\log n}\Big)\Big| \\
&&~~~~\le \max_{0\le u,v \le 1, |u-v|\le 1/\log n} |W(u)-W(v)| {\rightarrow} 0~~~~a.s.
\eestar
Now (\ref{th4.1eq8}) follows and the proof of  Theorem \ref{th4.1} is complete.
\end{proof}

\begin{proof}[Proof of Theorem \ref{th4.2}] Choose $t_n=\sqrt{n}$ for all $n\ge 1$. Since  $\check{a}_n\sim  n^{p}\Gamma(1-p)$,
it follows from (\ref{varcheck}) that
\bestar
\check{s}_n^2:=\mbox{Var}(\check{M}_n)=\check{a}_n^2 \mbox{Var}(\check{S}^*(n)) \sim \frac{n^{2p+1}}{2p+1}(\Gamma(1-p))^2\check{\sigma}^2.
\eestar
From the proof of Theorem \ref{th4.1}, we have
$\mathbb{P}(Z_n\ne X_n~i.o.)=0$ and hence by (\ref{checks1}) and Lemmas \ref{deltajn} and \ref{lemma7},
\be
\frac{1}{\sqrt{n}} (\check{S}^*(n)-\check{S}(n))= \sum_{j=1}^n \frac{\Delta_j(n)}{\sqrt{n}} (Z_j-X_j)\rightarrow 0~~a.s. \label{diffe2}
\ee
Then we can get Theorem \ref{th4.2} by using the same
arguments as in the proof of (\ref{th4.1eq1}) and (\ref{th4.1eq2}).
\end{proof}

\begin{proof}[Proof of (\ref{add})]  By (\ref{th4.1eq8}), we have
\bestar
&&\frac{1}{\sqrt{\log n}} \max_{1/n \le t\le 1}\Big|\frac{\hat{S}([nt])-m_1[nt]}{\sigma \sqrt{n}}-\sqrt{t} W(\log (nt))\Big|\\
 &&~~~~~~\le \frac{1}{\sqrt{\log n}} \max_{1/n \le t\le 1}\Big|\frac{\hat{S}([nt])-m_1[nt]}{\sigma \sqrt{nt}}-{W}(\log (nt))\Big| \stackrel{P}{\rightarrow} 0.
\eestar
Write $a_n=\log n/\log\log n$, then $a_n/\log\log n\rightarrow \infty$ and
\bestar
e^{-a_n/2}=o(e^{-\log \log n})=o((\log n)^{-1}).
\eestar
Theorem 1.2.1 in \cite{CR1981} implies that
\bestar
\frac{\sup_{0\le t\le \log n}\sup_{0\le s\le a_n} |W(t+s)-W(t)|}{\sqrt{4a_n\log\log \log n}}\rightarrow 1~~a.s.
\eestar
Hence applying the law of the iterated logarithm for $\{W(t), t\ge 0\}$ gives
\bestar
&&\frac{1}{\sqrt{\log n}}\sup_{1/n\le t\le 1}|\sqrt{t} (W(\log (nt))-W(\log n))|\\
&\le& 2e^{-a_n/2}\frac{\sup_{0\le t\le \log n}|W(t)|}{\sqrt{\log n}}+\sup_{e^{-a_n}\le t\le 1}\frac{|W(\log (nt))-W(\log n)|}{\sqrt{\log n}}\\
&\le& o(1) \frac{\sup_{0\le t\le \log n}|W(t)|}{(\log n)^{3/2}}+\frac{\sup_{0\le t\le \log n}\sup_{0\le s\le a_n} |W(t+s)-W(t)|}{\sqrt{\log  n}}
\rightarrow 0~~a.s.
\eestar
Now the desired result (\ref{add}) follows.
\end{proof}

\section*{Acknowledgment}
 The research was supported by NSFC (No. 11671373).


\end{document}